\documentclass{amsart}

\usepackage[all]{xy}    
\usepackage{graphicx}
\usepackage{amssymb}
\usepackage{mathrsfs}
\usepackage{multirow}
\usepackage{float}
\usepackage{tikz-cd}
\usepackage{adjustbox}
\usepackage{amsthm}
\usepackage{accents}
\usepackage{mathtools}
\usepackage{enumitem}

\usepackage{tikz}
\usetikzlibrary{matrix,arrows}

\usepackage[numbers,sort&compress]{natbib} 
\usepackage[bookmarksnumbered, bookmarksopen,
colorlinks,citecolor=blue,linkcolor=blue]{hyperref}

\newcounter{RomanNumber}
\newcommand{\MyRoman}[1]{\setcounter{RomanNumber}{#1}\Roman{RomanNumber}}

\newcommand{\qqed}{\hfill\Box}

\newtheorem{theorem}{Theorem}[section]
\newtheorem{lemma}[theorem]{Lemma}

\theoremstyle{definition}

\newtheorem{proposition}[theorem]{Proposition}
\newtheorem{corollary}[theorem]{Corollary}

\theoremstyle{remark}
\newtheorem{remark}[theorem]{Remark}

\theoremstyle{notation}

\numberwithin{equation}{section}

\begin{document}

\title{Homotopy of gauge groups over high dimensional manifolds}

\author{Ruizhi Huang}
\address{Institute of Mathematics, Academy of Mathematics and Systems Science, Chinese Academy of Sciences, Beijing, China, 100190}
\urladdr{https://sites.google.com/site/hrzsea/}

\email{huangrz@amss.ac.cn}
\thanks{The author was supported by Postdoctoral International Exchange Program for Incoming Postdoctoral Students under Chinese Postdoctoral Council and Chinese Postdoctoral Science Foundation.
He was also supported in part by Chinese Postdoctoral Science Foundation (Grant nos. 2018M631605 and 2019T120145), and National Natural Science Foundation of China (Grant no. 11801544). }


\numberwithin{theorem}{section}
\begin{abstract}
The homotopy theory of gauge groups has received considerable attention in recent decades. In this work, we study the homotopy theory of gauge groups over some high dimensional manifolds. To be more specific, we study gauge groups of bundles over $(n-1)$-connected closed $2n$-manifolds, the classification of which was determined by Wall and Freedman in the combinatorial category. We also investigate the gauge groups of the total manifolds of sphere bundles based on the classical work of James and Whitehead. Furthermore, other types of $2n$-manifolds are also considered. In all the cases, we show various homotopy decompositions of gauge groups. The methods are combinations of manifold topology and various techniques in homotopy theory.
\end{abstract}

\maketitle

\section{Introduction}
Let $G$ be a topological group and $P$ be a principal $G$-bundle over a space $X$. The \textit{gauge group} of $P$, denoted by $\mathcal{G}(P)$, is the group of $G$-equivariant automorphisms of $P$ that fix $X$. The topology of gauge groups and their classifying spaces plays a crucial role in mathematical physics and the geometry of manifolds. There is the remarkable theory of Donaldson \cite{Donaldson86} who applied the topology of gauge groups to study the differential structures of $4$-manifolds. From an algebraic topological point of view, Cohen and Milgram \cite{CM94} pointed out that the main question is to understand the homotopy types of gauge groups and related moduli spaces of connections. 

The homotopy theory of gauge groups has received considerable attention in recent decades.  
A fundamental result of Crabb and Sutherland \cite{CS00} claims that though there may be infinitely many isomorphism classes of principal $G$-bundles over $X$, their gauge groups have only finitely many distinct homotopy types. The explicit classifications of homotopy types of gauges groups were investigated particularly for $S^4$ by Kono \cite{Kono91}, Hasui-Kishimoto-Kono-Sato \cite{HKKS16}, Theriault \cite{Theriault17}, etc, and more generally for $4$-manifolds by Theriault \cite{Theriault10}, So \cite{So16} and So-Theriault \cite{ST19}. Moreover, Theriault \cite{Theriault10} and West \cite{West17} studied the homotopy decompositions of gauge groups over surfaces. In  higher dimensions, the author \cite{Huang19} studied the homotopy of gauge groups for non-simply connected $5$-manifolds, Membrillo-Solis \cite{MS19} for $7$-dimensional manifolds, Hamanaka-Kono \cite{HK07} for $S^6$ and Hamanaka-Kaji-Kono \cite{HKK08} for $S^8$. Additionally, Kishimoto-Kono-Tsutaya \cite{KKT13, KKT14} studied the general cases of even dimensional spheres. 

In this paper, we study the homotopy types of gauge groups over some high dimensional manifolds. We are mainly concerned with manifolds of the following three types:
\begin{itemize}[itemindent=11pt]
\item[\textbf{\qquad Type A:}] (Wall \cite{Wall62}) $(n-1)$-connected closed oriented {\it combinatorial} $2n$-manifolds in the sense of (page 181, \cite{Wall62});
\item[\textbf{Type B:}] (James-Whitehead \cite{JW54}) The total spaces of oriented sphere bundles of real vector bundles over spheres with cross sections;  
\item[\textbf{Type C:}] Other highly connected closed oriented $2n$-manifolds.
\end{itemize}

The basic idea to study the homotopy types of gauge groups over these manifolds, as applied in the work of Theriault \cite{Theriault10} and So \cite{So16}, is to decompose them into smaller and simpler pieces. It is a classical result that the isomorphism classes of $G$-principal bundles over a compact manifold $M$ are classified by the set of homotopy classes of classifying maps $[M, BG]$. Hence we may also denote the gauge group $\mathcal{G}(P)$ of $P$ with $[P]=\alpha\in [M, BG]$ by $\mathcal{G}_\alpha(M)$ to emphasize the base manifold $M$. By \cite{Gottlieb72} or \cite{AB83}, there is a homotopy equivalence
\[
B\mathcal{G}_\alpha(M)\simeq {\rm Map}_\alpha(M, BG)
\]
between the classifying space of $\mathcal{G}_\alpha(M)$ and the connected component of free mapping space ${\rm Map} (M, BG)$ that contains the classifying map representing $\alpha$. Hence from now on we do not distinguish these two spaces since we only care about their homotopy types.
Based on Proposition \ref{Gsplit} as a generalisation of the results of Theriault \cite{Theriault10} and So \cite{So16}, the relations between the homotopy types of the gauge groups and of the base manifolds can be summarised in the following slogan:
\begin{quote}
\textit{The homotopy type of the suspension of the base manifold controls the homotopy types of gauge groups, and moreover, the wedge decomposition of the suspended manifold leads to a product decomposition of gauge groups.
\textit}
\end{quote}
We now turn to discuss the three cases. For Type $A$, the classification of $(n-1)$-connected closed oriented $2n$-manifolds was studied by Wall \cite{Wall62}. Let $M$ be such a manifold of \textit{rank $m$}, i.e., $H^n(M; \mathbb{Z})\cong \oplus_{i=1}^{m}(\mathbb{Z})$. Wall studied the triple $(H^n(M;\mathbb{Z}), I_{M}, \beta_M)$ as the combinatorial invariant of $M$, where $I_{M}$ is the intersection form $I_M: H^n(M)\otimes H^n(M)\rightarrow \mathbb{Z}$, and 
$\beta_{M}: H^n(M;\mathbb{Z})\rightarrow \pi_{n-1}(SO(n))$
is a function defined by first realizing the cohomology classes of $H^n(M;\mathbb{Z})$ by embeddings of $n$-spheres in $M$ and then by taking the clutching functions of the normal bundles of the $n$-spheres in $M$. In general $\beta_{M}$ is not a homomorphism due to the complexity of the intersection form. However, we can show that the composition  
\[
\chi: H^n(M)\stackrel{\beta_M}{\rightarrow}  \pi_{n-1}(SO(n))\stackrel{J}{\rightarrow} \pi_{2n-1}(S^n)\stackrel{\Sigma}{\rightarrow} \pi_{n-1}(\mathbb{S})
\]
is a homomorphism, where $J$ is the classical $J$-homomorphism and $\mathbb{S}$ is the sphere spectrum. 
This homomorphism $\chi$ determines the homotopy structure of $\Sigma M$. Indeed, it is clear that ${\rm Im}\chi$ is a subgroup of ${\rm Im}\mathbb{J}$, the image of the stable $J$-homomorphism. The famous work of Adams \cite{Adams66} and Quillen \cite{Quillen71} determines ${\rm Im}\mathbb{J}$ (Theorem \ref{AdamsQuillen}), based on which we can determine the homotopy classification of $\Sigma M$ parametrized by a single index. We then deduce our first main result on homotopy decompositions of gauge groups, which generalizes the corresponding decomposition of Theriault \cite{Theriault10} for the case when $n=2$. Basically we see that the homotopy type of a gauge group is completely determined by the index.

\begin{theorem}\label{generalMn-12nintro}
Let $M$ be an $(n-1)$-connected closed oriented $2n$-manifold ($n\geq 2$) of rank $m$.
Let $G$ be a connected topological group with $\pi_{n-1}(G)\cong \pi_{n}(G)=0$.
Define the number ${\rm ind}(M)$ to be the index of the subgroup ${\rm Im}\chi$ in ${\rm Im}\mathbb{J}$ when ${\rm Im}\chi$ is not the trivial group, or $0$ otherwise.
Then we have the following homotopy equivalences for any $\alpha\in \pi_{2n-1}(G)\cong [M, BG]$,
\begin{itemize}
\item if $n\equiv 3, 5, 6, 7~{\rm mod}~8$,
\[\Sigma M\simeq S^{2n+1} \vee\bigvee_{i=1}^{m} S^{n+1}, ~~~ \ \ \mathcal{G}_\alpha(M)\simeq \mathcal{G}_\alpha(S^{2n})\times \prod_{i=1}^{m}\Omega^{n}G;\]
\item if $n\equiv 1, 2~{\rm mod}~8$, and $m\geq 2$,
\[\Sigma M\simeq \Sigma X_1\vee \bigvee_{i=1}^{m-1} S^{n+1},~~~ \ \ \mathcal{G}_\alpha(M)\simeq \mathcal{G}_\alpha(X_1)\times \prod_{i=1}^{m-1}\Omega^{n}G,\]
where $X_1=S^n\cup_ge^{2n}$ is such that $\Sigma g \equiv {\rm ind}(M)~{\rm mod}~2$.
If further we localize away from $2$,
\[\Sigma M\simeq S^{2n+1} \vee\bigvee_{i=1}^{m} S^{n+1}, ~~~ \ \ \mathcal{G}_\alpha(M)\simeq \mathcal{G}_\alpha(S^{2n})\times \prod_{i=1}^{m}\Omega^{n}G;\]
\item if $n=4s$, and $m\geq 2$,
\[\Sigma M\simeq \Sigma X_2\vee \bigvee_{i=1}^{m-1} S^{n+1},~~~ \ \ \mathcal{G}_\alpha(M)\simeq \mathcal{G}_\alpha(X_2)\times \prod_{i=1}^{m-1}\Omega^{n}G,\]
where $X_2=S^n\cup_ge^{2n}$ is such that $\Sigma g \equiv {\rm ind}(M)~{\rm mod}~d_s$, and $d_s$ is the denominator of $B_s/4s$ for $B_s$ the $s$-th Bernoulli number defined by 
\[\frac{z}{e^z-1}=1-\frac{1}{2}z-\sum\limits_{s\geq 1} B_s \frac{z^{2s}}{(2s)!}.\] 
\end{itemize}
\end{theorem}

For Type $B$, we consider the gauge groups over the total space of a spherical bundle over a sphere
\[S^q\stackrel{i}{\longrightarrow} E\stackrel{\pi}{\longrightarrow} S^n,\]
which admits a cross section, i.e., there exists a map $s: S^n\rightarrow E$ such that $\pi \circ s={\rm id}$.
It is known that $E$ has three cells as a $CW$ complex.
Suppose further that $(E,\pi)$ is the sphere bundle of some oriented real vector bundle. The homotopy types, or even the fibrewise homotopy types, of such bundles were classified by James and Whitehead \cite{JW54} via constructions using the $J$-homomorphism and the non-exact $EHP$-sequence. The first observation is that, due to the existence of a cross section, the clutching function $\zeta\in \pi_{n-1}(SO(q+1))$ of $(E,\pi)$ is the image of some class $\xi\in \pi_{n-1}(SO(q))$ through the natural inclusion $i: SO(q)\hookrightarrow SO(q+1)$. Then it can be showed that $\Sigma E\simeq S^{n+1} \vee {\rm Th}(E)$, 
where ${\rm Th} (E)$ is the Thom space of the bundle and 
${\rm Th} (E)\simeq \Sigma (S^q\cup_{J(\xi)} e^{q+n})$.
We can now state our second main result.
\begin{theorem}\label{spheregaugeintro}
Let 
$S^q\stackrel{i}{\longrightarrow} E\stackrel{\pi}{\longrightarrow} S^n$
be the sphere bundle of an oriented real vector bundle
which admits a cross section. Let $G$ be a connected topological group such that $\pi_{n-1}(G)=0$.
Then if $n$ and $q\geq 2$ we have a homotopy equivalence for any $\alpha\in [E, BG]$
\[\mathcal{G}_\alpha(E)\simeq \mathcal{G}_\alpha(S^q\cup_{J(\xi)} e^{q+n})\times \Omega^n G.\]

If further $\pi_{q-1}(G)=0$ and $J(\zeta)=0$, then when $n\leq 2q-1$
\[\mathcal{G}_\alpha(E)\simeq \mathcal{G}_\alpha(S^{q+n})\times \Omega^n G\times \Omega^q G.\]
\end{theorem}
The second part of the last theorem is proved based on a slightly stronger version (Lemma \ref{spheresplit3}) of a result of James-Whitehead (Theorem $1.11$ of \cite{JW54}), which deals with the triviality of the sphere bundle. 

For Type $C$, we consider the gauge groups of $E_7$-bundles over $4$-connected $12$-manifolds, and also of $E_8$-bundles over $6$-connected $16$-manifolds. In these two cases, the isomorphism classes of bundles are both classified by $\mathbb{Z}$ (Lemma \ref{Zindex}). With many homotopy theoretical techniques and computations, we can analyze the homotopy type of the suspension of such $M$, and then obtain our last theorem on the homotopy decompositions of the related gauge groups. 
\begin{theorem}\label{lasttheoremintro}
Let $M$ be an $(n-2)$-connected closed oriented $2n$-manifold with $H^{n-1}(M; \mathbb{Z})\cong \oplus_{i=1}^{m}\mathbb{Z}$ and $H^{n}(M;\mathbb{Z})=0$.
Then
\[M\simeq \big(\bigvee_{i=1}^{c}\Sigma^{n-3}\mathbb{C}P^{2}\vee\bigvee_{j=1}^{m-c}(S^{n-1}\vee S^{n+1})\big)\cup e^{2n}\]
for some non-negative number $c$ with $0\leq c\leq m$ and there are homotopy equivalences 
\begin{itemize}
\item of $E_7$-gauge groups when $n=6$
\[\mathcal{G}_k(M)\simeq \mathcal{G}_k(Z)\times \prod_{i=1}^{c-1} \Omega^3{\rm Map}^\ast(\mathbb{C}P^2, E_7) \times \prod_{i=1}^{m-c-1} \Omega^5 E_7
\times \prod_{i=1}^{m-c}\Omega^7 E_7,\]
where $Z\simeq \big(\Sigma^3\mathbb{C}P^2 \vee S^5\big)\cup_f e^{12}$ for some $f$;
\item and of $E_8$-gauge groups when $n=8$
\[\mathcal{G}_k(M)\simeq \mathcal{G}_k(Z)\times \prod_{i=1}^{c-2} \Omega^5{\rm Map}^\ast(\mathbb{C}P^2, E_8) \times \prod_{i=1}^{m-c-3} \Omega^7 E_8
\times \prod_{i=1}^{m-c-1}\Omega^9 E_8,\]
where $Z\simeq \big(\bigvee_{i=1}^{2} \Sigma^5\mathbb{C}P^2\vee \bigvee_{i=1}^{3}S^7\vee S^{9}\big)\cup_g e^{16}$ for some $g$.
\end{itemize}
If further we localize away from $2$, then in both cases we have the homotopy equivalences
\[\mathcal{G}_k(M)\simeq\mathcal{G}_k(S^{2n})\times \prod_{i=1}^{m}(\Omega^{n-1} E_i\times \Omega^{n+1} E_i),\]
where $(n,i)=(6, 7)$ or $(8, 8)$.
\end{theorem}

Theorems \ref{generalMn-12nintro}, \ref{spheregaugeintro} and \ref{lasttheoremintro} are useful for the study of the homotopy exponent problem of gauge groups. Among others, a typical result (Proposition \ref{expformu4}) in Section \ref{expsec} states that under certain conditions for our manifold $M$ in the theorems
\[
{\rm exp}_p(\mathcal{G}_\alpha(M))\leq p^{l(G)+{\rm ord}(\alpha)},
\]
where $2l(G)+1$ is the largest dimension of the spheres in the rational homotopy decomposition of $G$, and 
${\rm ord}(\alpha)$ denotes the $p$-primary order of the classifying map~$\alpha$.

Let us close the introduction with some comments. 
Besides the obvious applications to the homotopy and (co)homology of gauge groups, the results in this paper also inspire several further problems. Firstly, it is natural to consider other types of manifolds. In some sense, the methods in the paper provide a general framework to study the homotopy of gauge groups over high dimensional manifolds. Secondly, it is reasonable to apply geometrical tools to study gauge groups. For instance, it is an interesting question to ask how we can trace the change of homotopy type of gauge group if we apply a surgery on the underlying manifold. Finally, for the Type $B$ case, James-Whitehead indeed determined the complete classification of sphere bundles over spheres in terms of some types of obstructions defined by the $J$-homomorphism. Hence, it is an interesting question
whether the homotopy types of the gauge groups of the total spaces can be classified in terms of these obstructions for fixed structure group $G$.

Our paper is organized as follows. 
In Section \ref{MBGsec+Gsplitsec}, we classify principal bundles over highly connected closed manifolds with structure groups having certain properties, and then prove Proposition \ref{Gsplit}.
In Section \ref{complexsec}, we consider the homotopy suspensions of $CW$-complexes with the same homologies as our Type $A$ manifolds. The methods in this section, as suggested by our slogan, are general and interesting in their own right and will be used in the rest of the paper.
In Section \ref{n-12nsec}, we prove Theorem \ref{generalMn-12nintro} based on Wall's work on $(n-1)$-connected $2n$-manifolds. We also give some examples in this section.
In Section \ref{spherebundlesec}, we prove Theorem \ref{spheregaugeintro} based on the work of James and Whitehead on sphere bundles over spheres. 
In Section \ref{other2nsec}, we prove Theorem \ref{lasttheoremintro}.
Section \ref{expsec} is devoted to applications on the homotopy exponents.

\section{Principal bundles over highly connected closed manifolds and their gauge groups}\label{MBGsec+Gsplitsec}
In this section, we consider principal bundles and their gauge groups over highly connected closed manifolds. Let $M$ be a $k$-connected closed oriented $m$-manifold, and $M_i$ be its $i$-skeleton.
\begin{lemma}\label{Gindex}
Let $G$ be a connected topological group with $\pi_{i}(G)=0$ for $k\leq i\leq m-k-1$. Then $[M, BG]\cong \pi_{m-1}(G)$.
\end{lemma}
\begin{proof}
The homotopy cofibre sequence $S^{m-1}\stackrel{f}{\rightarrow} M_{m-1} \rightarrow M \rightarrow S^{m}\stackrel{\Sigma f}{\rightarrow} \Sigma M_{m-1}$
gives the the following exact sequence
\[[M_{m-1}, BG]\leftarrow [M, BG]\leftarrow [S^{m}, BG]\stackrel{(\Sigma f)^\ast}{\leftarrow} [\Sigma M_{m-1}, BG].\]
Since $M$ is $k$-connected, $M_{m-1}= M_{m-k-1}$ by Poincar\'{e} duality. Then we see $[M_{m-1}, BG]\cong [\Sigma M_{m-1}, BG]=0$. Hence the lemma follows. 
\end{proof}

Let us develop a general homotopy decomposition of gauge groups.
\begin{proposition}[cf. Proposition $2.1$ of \cite{Theriault10} and Lemma $2.3$ in \cite{So16}]\label{Gsplit}
Let $M$ be an oriented $m$-dimensional $k$-connected closed manifold, and $G$ be a connected topological group with $\pi_{i}(G)=0$ for $k\leq i\leq m-k-1$. Suppose that there exists a CW-complex $Y$ with $[Y, BG]=0$, and a map $\phi: Y\rightarrow M$ 
such that $\Sigma \phi$ admits a left homotopy inverse. Then 
\[\Sigma M\simeq \Sigma Y\vee\Sigma X,\]
where $X$ is the homotopy cofibre of $\phi$, and there is a homotopy equivalence
\[\mathcal{G}_\alpha(M)\simeq \mathcal{G}_\alpha(X)\times {\rm Map}^{\ast}(Y, G),\]
where $\alpha\in \pi_{m-1}(G)$.
\end{proposition}
\begin{proof}
The proof is similar to that of Lemma $2.3$ of \cite{So16}. The homotopy cofibre sequence $Y\stackrel{\phi}{\rightarrow} M \rightarrow X \rightarrow \Sigma Y\stackrel{\Sigma \phi}{\rightarrow} \Sigma M$
gives the the following exact sequence
\[[Y, BG]\leftarrow[M, BG]\leftarrow [X, BG]\leftarrow [\Sigma Y, BG]\stackrel{(\Sigma \phi)^\ast}{\leftarrow} [\Sigma M, BG].\]
By assumption we have $[Y, BG]=0$ and $(\Sigma \phi)^\ast$ is an epimorphism. Hence $[X, BG]\cong [M, BG]\cong \pi_{m-1}(G)$.
Then by considering the natural homotopy fibre sequence
\[
{\rm Map}^{\ast}_\beta(Z, G)\rightarrow  \mathcal{G}_\beta(Z)\rightarrow G\rightarrow {\rm Map}^{\ast}_\beta(Z, BG)\rightarrow {\rm Map}_\beta(Z, BG)\stackrel{{\rm ev}}{\rightarrow} BG
\]
for any class $\beta\in [Z, BG]$,
we have the following homotopy commutative diagram
\[
 \begin{tikzcd}
 \ast  \ar{r}  \ar{d}  & \Omega{\rm Map}^{\ast}_\alpha(M, BG) \ar[equal]{r} \ar[d]  & \Omega{\rm Map}^{\ast}_\alpha(M, BG)  \ar[d]  \\
 \mathcal{G}_\alpha(X) \ar{r}  \ar[equal]{d}    & \mathcal{G}_\alpha(M)  \ar{r} \ar{d} & {\rm Map}^{\ast}(\Sigma Y, BG)  \ar[d] \ar[bend right, dashed]{u}[swap]{\iota} \\
  \mathcal{G}_\alpha(X) \ar{r}  \ar{d}    & G \ar{r} \ar{d} & {\rm Map}^{\ast}_\alpha(X, BG) \ar{d}\\
  \ast  \ar{r}   &   {\rm Map}^{\ast}_\alpha(M, BG) \ar[equal]{r}   &  {\rm Map}^{\ast}_\alpha(M, BG)  
 \end{tikzcd}
\]
where the rows and columns are homotopy fibre sequence, and the map $\iota$ is a right homotopy inverse. Hence, the fibre sequence in the second row splits and the lemma follows.
\end{proof}

\section{Suspension splittings of certain complexes and their gauge groups}\label{complexsec}
In this section, we prove a suspension splitting of a complex $Z$ with $H_\ast(Z;\mathbb{Z})\cong H_\ast(M; \mathbb{Z})$, where $M$ is an $(n-1)$-connected closed oriented $2n$-manifold of rank $m$ ($m\geq 2$), that is, $H_2(M;\mathbb{Z})\cong \oplus_{i=1}^m \mathbb{Z}$. We consider this general case first rather than the special one for manifolds, since the method in this section does not depend on the manifold structure, and then may be of independent interest. In particular, it will be used often in the remaining sections of the paper.

First notice that $Z$ admits a cell decomposition $Z\simeq \bigvee_{i=1}^{m} S^{n}\cup_{f} e^{2n}$
with the attaching map $f$, the homotopy class of which lies in $\pi_{2n-1}(\bigvee_{i=1}^{m} S^{n})$. 
Recall that the stable homotopy groups of the sphere spectrum $\mathbb{S}$ are torsion abelian groups in positive dimensions and therefore admit decompositions of the form 
\begin{equation}\label{stablepisdeceq}
\pi_{n-1}(\mathbb{S})\cong \mathbb{Z}/k_1\oplus \mathbb{Z}/k_2\oplus \cdots\oplus \mathbb{Z}/k_r  
\end{equation}
such that $k_i|k_{i+1}$ for $1\leq i\leq r$.
By the Freudenthal suspension theorem, we have the surjective suspension morphism 
\[\Sigma: \pi_{2n-1}(\bigvee_{i=1}^{m} S^{n}) \twoheadrightarrow \bigoplus_{i=1}^{m}\pi_{n-1}(\mathbb{S}).\] 
Hence, we may write $f\simeq\sum\limits_{i=1}^{m}(a_{i,1}+a_{i,2}+\ldots +a_{i,r})+\omega$, where $\omega$ is in the kernel of $\Sigma$, and $\Sigma(a_{i,j})$ lies in the $j$-th component $\mathbb{Z}/k_j$ in the decomposition (\ref{stablepisdeceq}) of the $i$-th sphere. Notice that the choice of $a_{i,j}$ may be not unique. We may use a matrix to represent the homotopy class
\[\Sigma f = B, ~~~~ B=(b_{i,j})_{m\times r},~~~b_{i,j}=\Sigma(a_{i,j}).\]

Let $\{a, b\}\subseteq \{1, 2, \ldots m\}$. As in Lemma $2.5$ of \cite{So16}, we can define maps 
\[\phi_{a,b}: \bigvee_{i=1}^{m} S_i^{n}\rightarrow \bigvee_{i=1}^{m} S_i^{n}\]
such that its restriction on $S_a^{n}\vee S_b^{n}$ is the composition 
\[
S_a^{n}\vee S_b^{n}\stackrel{\sigma \vee {\rm id}}{\longrightarrow} S_a^{n}\vee S_a^{n}\vee S_b^{n} \stackrel{{\rm id}\vee \nabla}{\longrightarrow}  S_a^{n}\vee S_b^{n},
\]
with $\sigma$ the comultiplication and $\nabla$ the folding map, while $\phi_{a,b}$ leaves other factors fixed. 
Define $\tilde{f}_{a,b}=\phi_{a,b}\circ f$ and $ Z^\prime$ as its homotopy cofiber.
Then there is a homotopy commutative diagram
\[
 \begin{tikzcd}
S^{2n-1} \ar[equal]{d}  \ar{r}{f} & \bigvee_{i=1}^{m} S_i^{n}  \ar[r] \ar{d}{\phi_{a,b}} & Z \ar{d}{\tilde{\phi}_{a,b}}[swap]{\simeq} \\
S^{2n-1} \ar{r}{\tilde{f}_{a,b}}  & \bigvee_{i=1}^{m} S_i^{n}  \ar[r]   & Z^\prime,
 \end{tikzcd}
\]
where $\tilde{\phi}_{a,b}$ is a homotopy equivalence since it induces isomorphism on homology. It is easy to check that the effect of the map $\phi_{a,b}$ on $\Sigma f$ is an elementary row operation on $B$ given by adding the $a$-th row to the $b$-th row
\[\Sigma \tilde{f}_{a,b}= (I+E_{b,a})B,\] 
where $I$ is the identity matrix and $E_{b,a}$ is the matrix with value $1$ at $(b,a)$ and zero elsewhere. Hence, we can perform row-addition transformations on $B$ without changing the homotopy type of $Z$. Similarly, by the degree $-1$ map $-{\rm id}: S_a^{n}\rightarrow S_a^{n}$ for $1\leq a\leq m$, we can also perform row-multiplying transformations by $-1$ on $B$ without changing the homotopy type of $Z$.
Hence since the row-switching transformations are always allowed, we then can apply row-addition transformations, row-switching transformations and row-multiplying transformations by $-1$ on $B$ without changing the homotopy type of $Z$. 

There are two remarks. First, we should remember that the elements of the $j$-th column of $B$ are from $\mathbb{Z}/d_j$; second, in order to preserve the homotopy type of $Z$, the column operations are not allowed since in that case the elements under summation lie in different direct summands.
\begin{lemma}\label{trianglelemma}
With above conditions and notation, there is a choice of the attaching map $f$ for $M$ such that 
the suspension $\Sigma f$ can be represented by a matrix $B$ of the form 
\[
  \begin{pmatrix}
  e_1  & \ast  & \cdots  & \ast \\ 
  0      &e_2   & \cdots  &\ast \\
 \cdots & \cdots  & \cdots  & \cdots\\
 0        &0& \cdots        &e_r \\
0  & 0  & \cdots  &0             \\ 
  \cdots & \cdots  & \cdots  & \cdots\\ 
  0  & 0  & \cdots  &0,          
 \end{pmatrix}
 ~~~{\rm or}~~
  \begin{pmatrix}
  e_1  & \ast  & \cdots  & \ast & \cdots&\ast\\ 
  0      &e_2   & \cdots  &\ast & \cdots&\ast \\
 \cdots & \cdots  & \cdots  & \cdots& \cdots&\ast \\
 0        &0 &\cdots         &e_m& \cdots&\ast  \\  
 \end{pmatrix}
\]
depending on $r\leq m$ or $r\geq m$.
\end{lemma}
\begin{proof}
Since the Euclidean algorithm for the great common divisor of numbers $(a_1,\ldots, a_n)$
can be interpreted as the row-addition transformations on the column vector $(a_1,\ldots, a_n)^{T}$,
we can use row-addition transformations to transform our original matrix $B$ of $\Sigma f$ to a new matrix $B_1^\prime$ such that the first column of $B_1^\prime$ is of the form $(0,\ldots, \pm e_1, \ldots,0)^T$ where $e_1$ is the great common divisor of the elements of the first column of $B$. Since we can also apply row-switching transformations and row-multiplying transformations by $-1$, we can further transform $B_1^\prime$ to $B_1$ with first column $(e_1,0, \ldots, 0)^T$. Repeating the process for the remaining rows and columns of $B_1$, we can achieve a matrix of the form as in the statement of the lemma. By our previous discussion, these transformations can be realized without changing the homotopy type of $Z$. Hence we have proved the lemma. 
\end{proof}

\begin{remark}\label{remarkgcd}
For any set $S\subset \mathbb{Z}/d$ $(d\neq 0)$, we can define \textit{the greatest common divisor} ${\rm g.c.d.}(S)$ of $S$ to be the minimum of ${\rm g.c.d.}(s_1, \ldots, s_l)$ with $(s_1, \ldots, s_l)$ running through all the representatives of $S$ in $\mathbb{Z}/d$ up to units.
In particular, the diagonal entries $e_1$ of the matrices in Lemma \ref{trianglelemma} can be chosen to be ${\rm g.c.d.}(\{b_{i,1}\}_{i=1}^{m})$. 

For instance, when $m=1$ and $S=\{s\}$ with $s\in \mathbb{Z}/d$, ${\rm g.c.d.}(S)$ is the minimum of $|x|$ among all $x$ such that the congruence class $[x]=cs$ for some unit $c\in \mathbb{Z}/d$. Indeed, it is easy to see that ${\rm g.c.d.}(S)$ is equal to the index of the subgroup generated by $s$ in $\mathbb{Z}/d$.
\end{remark}

By Lemma \ref{trianglelemma} and its proof, we can easily get the following splitting of $\Sigma Z$. 
\begin{corollary}\label{Zsplit}
Let $Z$ be an $(n-1)$-connected $CW$-complex ($n\geq 2$) with cell decomposition $Z\simeq \bigvee_{i=1}^{m} S^{n}\cup_{f} e^{2n}$.
Suppose the matrix representation $B$ of $\Sigma f$ has only $t$ non-zero columns after row operations and $t<m$. Then we have the homotopy equivalence
\[\Sigma Z\simeq \Sigma X\vee \bigvee_{i=1}^{m-t} S^{n+1},\]
where $X$ is the cofibre of the inclusion $\bigvee_{i=1}^{m-t} S^{n}\hookrightarrow Z$.  ~$\qqed$
\end{corollary}

\section{Suspension splittings of highly connected manifolds and their gauge groups}\label{n-12nsec}
In this section, $M$ will be an $(n-1)$-connected closed oriented $2n$-manifold of rank $m$. We will study its  suspension splitting based on its manifold structure and then develop the homotopy decompositions of the associated gauge groups.

\subsection{Recollection of homotopy aspects of $(n-1)$-connected $2n$-manifolds}\label{attachingsubsec}
The homeomorphism types and the homotopy types of closed $(n-1)$-connected $2n$-manifolds are classified by Wall \cite{Wall62} and Freedman \cite{Freedman82}. Let us recall some information we need here (\cite{Wall62}, \cite{DW03}).

First let us describe some particular aspects of the homotopy structure of $M$ as follows. By Poincar\'{e} duality, $M$ admits a cell decomposition 
\[M\simeq \bigvee_{i=1}^{m} S^{n}\cup_{f} e^{2n}.\]
By a result of Haefliger \cite{Haefliger61} if $n\geq 3$ each homology class $z\in H_n(M)\cong \oplus_{i=1}^m\mathbb{Z}$ is presented by an embedding of an oriented $n$-sphere $S^n$ in $M$. Then we get an oriented normal bundle $\nu$ of $S^n$ as the tubular neighborhood of $S^n$ in $M$ with orientation inherited from that of $M$. The bundle $\nu$ determines a homotopy class $\alpha\in \pi_{n-1}(SO(n))$ through a clutching function, and the function $\alpha_M: H_n(M)\rightarrow \pi_{n-1}(SO(n))$ sending $z$ to $\alpha$ is well defined.
Hence we obtain a function 
\begin{equation}\label{nspacebeta}
\beta: H^n(M)\stackrel{\pi_M}{\longrightarrow}H_n(M)\stackrel{\alpha_M}{\longrightarrow} \pi_{n-1}(SO(n)),
\end{equation}
where $\pi_M$ is the isomorphism given by Poincar\'{e} duality. 
For the case when $n=2$, we may define $\beta(x)=I_M(x\otimes x)y$, where $x\in H^2(M)$, $y\in \pi_1(SO(2))\cong\mathbb{Z}$ is the class of identity, and $I_M$ is the intersection form of $M$.
In general $\beta$ is not a homomorphism.
Let $[S_i^n]\in H_n(M)\cong H_{n}(\bigvee_{j=1}^{m}S_j^n)$ be the fundamental class of the $i$-th sphere in $\bigvee_{j=1}^{m}S_j^n$. It determines a homotopy class $\alpha_i\in \pi_{n-1}(SO(n))$. Consider the composition 
\[
f_i: S^{2n-1}\stackrel{f}{\rightarrow} \bigvee_{j=1}^{m}S_j^n  \stackrel{q_i}{\rightarrow} S_i^n
\]
with $q_i$ the projection onto the $i$-th component. The homotopy class of $f_i$ is indeed the image of $\alpha_i$ under the $J$-homomorphism of Whitehead \cite{Whitehead78} $J: \pi_{n-1}(SO(n))\rightarrow \pi_{2n-1}(S^n)$.
Recall that there is a natural commutative diagram up to sign
\begin{eqnarray}\label{JSigmadiagram}
\begin{tikzcd}
\pi_{n-1}(SO(n)) \ar{r}{J}  \ar[two heads]{d}{i_\ast} & \pi_{2n-1}(S^n) \ar[two heads]{d}{\Sigma} \\
\pi_{n-1}(SO(n+1)) \ar{r}{J}  & \pi_{2n}(S^{n+1}).
\end{tikzcd}
\end{eqnarray}
The $J$-homomorphism in the second row of the diagram is indeed in the stable range, i.e., 
$\mathbb{J}: \pi_{n-1}(SO)\rightarrow \pi_{n-1}(\mathbb{S})$,
whose image is known due to famous work of Adams and Quillen.
\begin{theorem}[\cite{Adams66}, \cite{Quillen71}]\label{AdamsQuillen}
\[
{\rm Im}\mathbb{J}=\left\{\begin{array}{cc}
0 & n\equiv 3, 5, 6, 7 \ {\rm mod} \ 8, \\
\mathbb{Z}/2 & n\equiv 1, 2\ {\rm mod} \ 8, n\neq 1, \\
\mathbb{Z}/d_s& n=4s,
\end{array}\right.
\]
where $d_s$ is the denominator of $B_s/4s$ for $B_s$ the $s$-th Bernoulli number.~$\qqed$
\end{theorem}

\subsection{Homotopy decompositions of gauge groups over $(n-1)$-connected $2n$-manifolds}
To access the gauge groups we need to analyze the homotopy type of $\Sigma M$ more carefully, based on the method in Section \ref{complexsec}. Indeed, we want to apply matrix transformations through the function $\beta$ (\ref{nspacebeta}) and $J$.

To that purpose, first recall that for the manifold $M$ we have the intersection form $I_M: H^n(M)\otimes H^n(M)\rightarrow \mathbb{Z}$ defined by $I_M(x\otimes y)=\langle xy, [M]\rangle$, where $[M]$ is the fundamental class of $M$. Then the function $\beta: H^n(M)\rightarrow  \pi_{n-1}(SO(n))$ (\ref{nspacebeta}) satisfies the relations:
\[
HJ\beta(x)= I_M(x\otimes x), \ \ x\in H^n(M),
\]
\begin{equation}\label{betasum}
\beta(x+y)=\beta(x)+\beta(y)+I_M(x\otimes y) \partial^\prime \iota_n, x, y \in H^n(M),
\end{equation}
where $H$ is the Hopf invariant, $\iota_n\in \pi_{n}(S^n)$ is given by the identity map and $\partial^\prime: \pi_{n}(S^n)\rightarrow \pi_{n-1}(SO(n))$ is the connecting homomorphism in the long exact sequence of homotopy groups for the fibration $SO(n)\stackrel{i}{\rightarrow} SO(n+1)\rightarrow S^n$.
Applying the composition of the $J$-homomorphism and the suspension $\Sigma $ to (\ref{betasum}), we have 
\begin{eqnarray*}
\Sigma J\beta(x+y)
&=&\Sigma J\beta(x)+\Sigma J\beta(y)+I_M(x\otimes y) \Sigma J\partial^\prime \iota_n\\
&=&\Sigma J\beta(x)+\Sigma J\beta(y)+I_M(x\otimes y) \Sigma (-[\iota_n,\iota_n])\\
&=&\Sigma J\beta(x)+\Sigma J\beta(y),
\end{eqnarray*}
where we used the fact that $-J\partial^\prime=[-,\iota_n]$ is a Whitehead product (e.g., see (1.3) of \cite{JW54}). This means the composition
\begin{equation}\label{chi}
\chi: H^n(M)\stackrel{\beta}{\rightarrow}  \pi_{n-1}(SO(n))\stackrel{J}{\rightarrow} \pi_{2n-1}(S^n)\stackrel{\Sigma}{\rightarrow} \pi_{n-1}(\mathbb{S})
\end{equation}
is a homomorphism. Hence the effects of matrix transformations used in Lemma \ref{trianglelemma} on $\Sigma M$ can be understood through $\chi$.
Indeed, by Theorem \ref{AdamsQuillen} and Diagram (\ref{JSigmadiagram}) we have that the image of 
$\chi=\Sigma J \beta=- \mathbb{J} i_\ast\beta$
is always a cyclic group. We then can choose a column vector $B=\{b_{i}\}_{m\times 1}$ to represent the suspension of the attaching map $f:S^{2n-1}\rightarrow \bigvee_{j=1}^{m}S_j^n$ of $M$ as in Section \ref{complexsec} such that $b_{i}\in \mathbb{Z}/d\cong {\rm Im}(\Sigma J)$ for some $d$. 
By Lemma \ref{trianglelemma} and Remark \ref{remarkgcd}, $B$ can be transformed to a vector of the form 
$(h, 0, \ldots, 0)^T$
such that $h={\rm g. c. d.}(\{b_{i}\}_{i=1}^{m})$ in the sense of Remark \ref{remarkgcd}.

To understand the transformations through $\chi$, let us write the attaching map $f:S^{2n-1}\rightarrow \bigvee_{j=1}^{m}S_j^n$ of $M$ as $f\simeq f_1+\cdots +f_m+\omega$,
where $f_i=J(\alpha_i)=J\beta([S_i^n]^\ast)$ is the $i$-th component, and $[S_i^n]^\ast$ is the dual of $[S_i^n]$ under Poincar\'{e} duality. Notice that $f_i$ corresponds to the $i$-th row of $B$. Then 
\[
\Sigma(f_a+f_b)=\chi([S_a^n]^\ast+[S_b^n]^\ast)=\chi([S_a^n]^\ast)+\chi([S_b^n]^\ast)=\Sigma f_a +\Sigma f_b,
\]
which means that the effect of an elementary row operation on $B$ by adding the $a$-th row to the $b$-th row can be achieved by the addition operation of the group $H^n(M)\cong \oplus_{i=1}^{m}\mathbb{Z}$. Hence, it follows that $h$ is indeed the index of the subgroup ${\rm Im}\chi$ in ${\rm Im}\mathbb{J}$ unless $h=0$.

We can now prove Theorem \ref{generalMn-12nintro}.
\begin{proof}[Proof of Theorem~\ref{generalMn-12nintro}] 
First by Lemma \ref{Gindex} $[M, BG]\cong \pi_{2n-1}(G)$ if $\pi_{n-1}(G)\cong \pi_{n}(G)=0$. Notice that ${\rm ind}(M)=h$ by the above discussions and Remark \ref{remarkgcd}.
The homotopy classification of $\Sigma M$ now follows from the above arguments and Theorem \ref{AdamsQuillen}, and then the corresponding splittings of gauge groups follows immediately from Proposition \ref{Gsplit}. 

To complete the proof, it remains to show that the two decompositions of gauge groups involving $X_1$ and $X_2$ in the theorem do not depend on the choices of the attaching map $g$. Indeed, let $X^{(1)}=S^n\cup_{g^{(1)}} e^{2n}$ and $X^{(2)}=S^n\cup_{g^{(2)}} e^{2n}$ such that $\Sigma g^{(1)}\simeq \Sigma g^{(2)}$. Similar to Lemma 2.12 of \cite{So16} we can consider the composition $h: S^{2n-1}\stackrel{\mu^\prime}{\longrightarrow} S^{2n-1}\vee S^{2n-1}\stackrel{g^{(1)}\vee g^{(2)}}{\longrightarrow} S^{n}\vee S^{n}$ and its homotopy cofiber $Y$, where $\mu^\prime$ is the co-multiplication. Then by Lemma \ref{trianglelemma}, $\Sigma Y\simeq \Sigma X^{(1)}\vee S^{n+1}\simeq \Sigma X^{(2)}\vee S^{n+1}$. It implies that $\mathcal{G}_\alpha (X^{(1)})\times \Omega^n G\simeq \mathcal{G}_\alpha (X^{(2)})\times \Omega^n G$ by Proposition \ref{Gsplit} .
\end{proof}

\subsection{Examples}\label{examples2n-1n} 
We specify some examples for bundles over particular manifolds and structure groups.
\subsubsection{Bundles of exceptional Lie groups}
We may consider several examples of bundles over exceptional Lie groups, which can be also studied through the stable homotopy groups of spheres directly.
Let us collect some information about the homotopy groups of spheres (\cite{Toda62}):
\[\pi_9(S^5)\cong \mathbb{Z}/2,  ~~\pi_{11}(S^6)\cong \mathbb{Z}, ~~\pi_{15}(S^8)\cong \mathbb{Z}\oplus \mathbb{Z}/120,\]
\[\pi_4(\mathbb{S})=\pi_5(\mathbb{S})=0, ~~ \pi_7(\mathbb{S})\cong \mathbb{Z}/240,\]
where $\Sigma: \pi_{15}(S^8)\twoheadrightarrow \pi_7(\mathbb{S})$ sends the generator $\sigma_8\in \mathbb{Z}$ representing the third Hopf invariant $1$ map to the generator $\sigma\in \mathbb{Z}/240$, and the generator in $\mathbb{Z}/120$ to $2\sigma$. Combining this information with Proposition \ref{Gsplit}, Corollary \ref{Zsplit}, Lemma \ref{Zindex} and the information on homotopy groups of $E_6$, $E_7$ and $E_8$, we can get the homotopy decompositions of associated gauge groups.

\begin{proposition}\label{examplen-12n}
Let $M$ be an $(n-1)$-connected closed oriented $2n$-manifold of rank $m$. We have
\begin{itemize}
\item when $n=5$ and $G=E_6$,
$\mathcal{G}_k(M)\simeq \mathcal{G}_k(S^{10})\times \prod_{i=1}^{m}\Omega^{5}E_6$; 
\item when $n=6$ and $G=E_7$, 
$\mathcal{G}_k(M)\simeq \mathcal{G}_k(S^{12})\times \prod_{i=1}^{m}\Omega^{6}E_7$; 
\item when $n=8$, $G=E_8$ and $\Sigma f$ is null-homotopic,
$\mathcal{G}_k(M)\simeq \mathcal{G}_k(S^{16})\times \prod_{i=1}^{m}\Omega^{8}E_8$;
\item when $n=8$, $G=E_8$ and $\Sigma f$ is not null-homotopic,
$\mathcal{G}_k(M)\simeq \mathcal{G}_k(X_\alpha)\times \prod_{i=1}^{m-1}\Omega^{8}E_8$, where $X_\alpha$ is a $2$-cell complex with one $8$-cell and one $16$-cell and an attaching map $\alpha\in \pi_{15}(S^8)$. ~$\qqed$
\end{itemize}
\end{proposition}

In the fourth case of Proposition \ref{examplen-12n}, the attaching map $\alpha$ of $X_\alpha$ is determined by the original $f$ and the matrix transformations of $\Sigma f$. By Theorem \ref{generalMn-12nintro}, we see that 
$\Sigma X_\alpha\simeq S^9\cup_h e^{17}$,
where $h$ is the index of the subgroup ${\rm Im}\chi$ in $\mathbb{Z}/240$ .

If further $M$ is \textit{almost parallelizable}, i.e., $M$ is parallelizable with one point deleted, then the ${\rm mod}~2$ Steenrod module structure on $H^\ast(M; \mathbb{Z}/2)$ is trivial due to the triviality of Stiefel-Whitney class of $M$ by the famous Wu formula (\cite{Wu55}, page $132$ of \cite{MS75}). Hence we have the following homotopy decomposition in this special case since $Sq^{8}$ detects $\sigma$:
\begin{proposition}\label{paraM16}
Let $M$ be a $7$-connected almost parallelizable closed oriented $16$-manifold of rank $m$. Then for the gauge group of the $E_8$-bundle over $M$ corresponding to $k\in [M, BE_8]$, we have the homotopy equivalence after localization away from $15$
\[
\hspace{3.85cm}\mathcal{G}_k(M)\simeq \mathcal{G}_k(S^{16})\times \prod_{i=1}^{m}\Omega^{8}E_8.
\hspace{3.85cm}\Box\] 
\end{proposition}

\subsubsection{Stable $Sp$-bundles}
Recall the Bott periodicity for symplectic groups for $q-1\leq 4r$:
\[\def\arraystretch{1.3}
\begin{array}{ c | c | c | c | c | c | c | c | c }
q~{\rm mod}~8 &  0 & 1 & 2 & 3 & 4 & 5 & 6 & 7 \\ \hline
\pi_q(Sp(r)) &0 &0 & 0 &\mathbb{Z} & \mathbb{Z}/2& \mathbb{Z}/2 &0 &\mathbb{Z}
\end{array}.
\]
\begin{proposition}\label{examplesp}
Let $M$ be an $(8n+1)$-connected closed oriented $(16n+4)$-manifold of rank $m$. For the $Sp(r)$-bundle corresponding to $k\in [M, BSp(r)]\cong \mathbb{Z}$ with $8n+1\leq 2r$, we have the homotopy equivalence 
\[
\mathcal{G}_k(M)\simeq \mathcal{G}_k(X)\times \prod_{i=1}^{m-1}\Omega^{8n+2}Sp(r),
\]
where $X_1=S^{8n+2}\cup_ge^{16n+4}$ is such that $\Sigma g=h$.
If further we localize away from $2$,
\[\hspace{3.2cm}
\mathcal{G}_k(M)\simeq \mathcal{G}_k(S^{16n+4})\times \prod_{i=1}^{m}\Omega^{8n+2}Sp(r).
\hspace{3.2cm}\Box\]
\end{proposition}

\subsubsection{Stable $Spin$-bundles}
Recall the Bott periodicity for Spin groups for $r\geq q+2\geq 4$:
\[\def\arraystretch{1.3}
\begin{array}{ c | c | c | c | c | c | c | c | c }
q~{\rm mod}~8 &  0 & 1 & 2 & 3 & 4 & 5 & 6 & 7 \\ \hline
\pi_q(Spin(r)) & \mathbb{Z}/2& \mathbb{Z}/2 &0 &\mathbb{Z}&0 &0 & 0 &\mathbb{Z} 
\end{array}.
\]
\begin{proposition}\label{examplespin}
Let $M$ be a $(4n+1)$-connected closed oriented $(8n+4)$-manifold of rank $m$. For the $Spin(r)$-bundle corresponding to $k\in [M, BSpin(r)]\cong \mathbb{Z}$ with $r\geq 8n+5$, we have the homotopy equivalence localized away from $2$
\[
\hspace{3.1cm}\mathcal{G}_k(M)\simeq \mathcal{G}_k(S^{8n+4})\times \prod_{i=1}^{m}\Omega^{4n+2}Spin(r).
\hspace{3.1cm}\Box
\]
\end{proposition}

\section{Sphere bundles over spheres and their gauge groups}\label{spherebundlesec}
In this section, we consider the gauge groups over the total space of a spherical bundle over a sphere
$S^q\stackrel{i}{\longrightarrow} E\stackrel{\pi}{\longrightarrow} S^n$,
which admits a cross section, i.e., there exists a map $s: S^n\rightarrow E$ such that $\pi \circ s={\rm id}$.
By \cite{JW54} we know that $E$ is a $3$-cell $CW$-complex with one $q$-cell, one $n$-cell and one $(q+n)$-cell.
As before we need to study the homotopy structure of $\Sigma E$ and then apply Proposition \ref{Gsplit}.

\subsection{Suspension splittings of sphere bundles over spheres with a section}
Let us suppose $n$, $q\geq 2$. Then we have a homotopy commutative diagram 
\[
\begin{tikzcd}
S^q \ar[hook]{r} \ar[equal]{d}  & E_{q+n-1} \ar{r} \ar[hook]{d}{j} & S^n \ar[equal]{d} \ar[bend left, dashed]{l}[swap]{s} \\
S^q \ar{r}     &   E \ar{r}{\pi}   & S^n \ar[bend left]{l}{s},
\end{tikzcd}
\]
where the first row is the cofibration for the $(q+n-1)$-skeleton $E_{q+n-1}$ of $E$, and the section $s$ factors through the injection $j$ by the cellular approximation. Hence, we have the homotopy equivalence $E_{q+n-1}\simeq S^q \vee S^n$, and $E\simeq (S^q\vee S^n)\cup_h e^{q+n}$ with the attaching map $h$ for the top cell.
We then also have another homotopy commutative diagram of cofibrations
\begin{eqnarray}\label{defeta}
\begin{tikzcd}
\ast \ar{r} \ar{d}  & S^n \ar[equal]{r} \ar[hook]{d}  &S^n \ar{d}[swap]{s} \\
S^{q+n-1} \ar{r}{h} \ar[equal]{d}  & S^q\vee S^n \ar{d} \ar{r}  & E \ar{d}[swap]{p}  \ar[bend right]{u}[swap]{\pi}\\
S^{q+n-1} \ar{r}{\eta}  &S^q \ar{r}  & X,
\end{tikzcd}
\end{eqnarray}
which defines $\eta$ as one of the components of $h$, the complex $X$ as the homotopy cofibre of $\eta$, and $p$ as the induced map. Since $\Sigma E$ is a co-$H$-space, we see that the suspension of the last column of Diagram (\ref{defeta}) splits, i.e., $\Sigma E\simeq S^{n+1}\vee \Sigma X$.
On the other hand, we can also get the suspension splitting of $\Sigma E$ using Thom complexes. Recall that for the spherical bundle $(E, \pi)$, the \textit{Thom complex} ${\rm Th}(E)$ of $(E, \pi)$ is just the homotopy cofibre of $\pi$ and we have the homotopy cofibre sequence
\[E\stackrel{\pi}{\rightarrow} S^n\rightarrow {\rm Th}(E) \stackrel{g}{\rightarrow} \Sigma E \stackrel{\Sigma \pi}{\rightarrow} S^{n+1},\]
which defines the map $g$.
Since $\Sigma \pi\circ \Sigma s\simeq {\rm id}$, we then obtain $\Sigma E\simeq S^{n+1} \vee {\rm Th}(E)$.
The complexes $\Sigma X$ and ${\rm Th}(E)$ are related. Indeed, by the above arguments the composition map
\[
T: ~S^{n+1}\vee {\rm Th}(E)\stackrel{\Sigma s\vee g}{\rightarrow} \Sigma E \stackrel{\mu^\prime}{\rightarrow} \Sigma E\vee \Sigma E \stackrel{\Sigma \pi \vee \Sigma p}{\rightarrow} S^{n+1}\vee \Sigma X,
\]
is a homotopy equivalence, where $\mu^\prime$ is the co-multiplication. Let $p_2: S^{n+1}\vee \Sigma X\rightarrow\Sigma X$ and $i_2:  {\rm Th}(E)\rightarrow S^{n+1}\vee {\rm Th}(E)$ be the canonical projection and injection respectively.
Then it is not hard to show that the composition 
\[
H^\ast(\Sigma X)\stackrel{{p_2}^\ast}{\rightarrow}H^\ast(S^{n+1}\vee {\rm Th}(E))\stackrel{T^\ast}{\rightarrow} H^\ast(S^{n+1}\vee {\rm Th}(E))\stackrel{i_2^\ast}{\rightarrow}H^\ast( {\rm Th}(E))
\] 
is an isomorphism.
Hence the composition ${\rm Th}(E)\stackrel{g}{\rightarrow}\Sigma E\stackrel{\Sigma p}{\rightarrow} \Sigma X$ is a homotopy equivalence.

\subsection{Sphere bundles of plane bundles over spheres}\label{subsec: spherebundle2}
Suppose $E$ is the sphere bundle of an oriented vector bundle over $\mathbb{R}$. In this case we can get a better description, that is, we can desuspend the homotopy equivalence ${\rm Th}(E)\simeq \Sigma X$. Also notice that we do not need the restriction on $q$ and $n$ if we use Thom complexes to study $\Sigma E$. The material in this subsection mainly follows the work of James and Whitehead on spheres bundles over spheres \cite{JW54}.

Indeed, as the sphere bundle of a $(q+1)$-plane bundle, $E$ is determined by a clutching function 
$\zeta \in \pi_{n-1}(SO(q+1))$,
and (Milnor; Proposition $29$ of \cite{FB15})
\begin{equation}\label{planeThom}
{\rm Th}(E)\simeq S^{q+1}\cup_{J(\zeta)} e^{q+n+1},
\end{equation}
where $J: \pi_{n-1}(SO(q+1))\rightarrow \pi_{n+q}(S^{q+1})$ is the $J$-homomorphism.
Further, the clutching function $\zeta$ is the image of the class of the identity under $\partial: \pi_{n}(S^n)\rightarrow \pi_{n-1}(SO(q+1))$, which is the connecting homomorphism of the principal bundle $SO(q+1)\rightarrow P(E)\rightarrow S^n$ of $E$.
Then by the following commutative diagram of the homotopy groups of three fibrations
\[
\begin{tikzcd}
& \pi_n(S^n) \ar{d}{\partial}   \ar[equal]{r}   & \pi_n(S^n) \ar{d}{\delta=0}\\
\pi_{n-1}(SO(q)) \ar{r}{j_\ast} & \pi_{n-1}(SO(q+1)) \ar{d} \ar{r} & \pi_{n-1}(S^q) \ar{d}{i_\ast} \\
&\pi_{n-1}(P(E))\ar{r}  &\pi_{n-1}(E),
\end{tikzcd}
\]
there exists some $\xi\in \pi_{n-1}(SO(q))$ such that $j_\ast(\xi)=\zeta$. Hence by (\ref{planeThom}) and  Diagram (\ref{JSigmadiagram}), we have 
\[
{\rm Th}(E) \simeq S^{q+1}\cup_{Jj_\ast(\xi)} e^{q+n+1} \simeq S^{q+1}\cup_{\Sigma J(\xi)} e^{q+n+1}
\simeq \Sigma (S^q\cup_{J(\xi)} e^{q+n}).
\]
Indeed by the discussion of Section $3$ of \cite{JW54}, we actually can choose $\eta$ in Diagram (\ref{defeta}) to be $J(\xi)$ and then 
$X\simeq S^q\cup_{J(\xi)} e^{q+n}$.
\subsection{Homotopy decompositions of gauge groups of sphere bundles}
In this subsection we prove Theorem \ref{spheregaugeintro}, for which we need a technical lemma.
Let $\mathbb{R}^{q+1}\rightarrow V\rightarrow S^n$ be an oriented $(q+1)$-plane bundle with the associated sphere bundle $S^{q}\rightarrow E\rightarrow B$. James and Whitehead \cite{JW54} have theoretically classified the homotopy type of sphere bundles of oriented vector bundles over spheres with cross sections. Moreover, they showed a sufficient condition for when $E\simeq S^q\times S^n$ without the assumption of the existence of a cross section.

\begin{lemma}[Theorem $1.11$ of \cite{JW54}]\label{spheresplit3}
Suppose $J(\zeta)=0$ and $n\leq 2q-1$. Then $E=S(V)\simeq S^q\times S^n$.
\end{lemma}
\begin{proof}
The proof here is due to the work of James and Whitehead, but partly differs from theirs since we do not apply their classification theorem. We try to organize the details of the proof since many of them are classical and interesting, and we  should also notice that the statement of the lemma is a little stronger than the theorem of James and Whitehead.  

The first step is to prove that the associated sphere bundle 
\begin{equation}\label{SVbundle}
S^{q}\stackrel{i}{\longrightarrow} E\stackrel{\pi}{\longrightarrow} S^{n},
\end{equation}
of $V$ admits a cross section, which is equivalent to proving that for the standard principal bundle
$
SO(q)\stackrel{j}{\longrightarrow} SO(q+1) \stackrel{p}{\longrightarrow} S^q,
$
we have $p_\ast(\zeta)=0$ where $\zeta\in \pi_{n-1}(SO(q+1))$ is the clutching function of $V$, that is, $j_\ast(\xi)=\zeta$ for some $\xi\in \pi_{n-1}(SO(q))$ and then $V$ is indeed a fibre bundle of $SO(q)$.
For this purpose, we recall a result of G. W. Whitehead (Theorem $5.1$ of \cite{Whitehead50}). 
Consider $f: S^{n-1} \times S^{m} \rightarrow S^{q}$ with $n+m\leq 3q-1$. Its restriction map $(\alpha,\beta): S^{n-1}\vee S^{m}\hookrightarrow S^{n-1} \times S^{m} \rightarrow S^{q}$ is called \textit{the type of $f$}. Whitehead showed that $H\circ F (f)\simeq (-1)^{q+1} (\alpha\ast \beta)$, 
where $F(f)$ is the Hopf construction of $f$, $H$ is the Hopf invariant and $\alpha\ast \beta$ is the join of the involved maps. In other words,
there exists a commutative diagram up to the sign $(-1)^{q+1}$
\[
\begin{tikzcd}
\lbrack S^{n-1}\times S^{m}, S^{q}\rbrack \ar{r}{F} \ar{d} &
\pi_{n+m}(S^{q+1}) \ar{d}{H}\\
\pi_{n-1}(S^{q}) \oplus \pi_{m}(S^{q})
\ar{r}{\ast}  &  \pi_{n+m}(S^{2q+1}). 
\end{tikzcd}
\]
Since $\zeta \in \pi_{n-1}(SO(q+1))$ determines a map $\tilde{\zeta}: S^{n-1}\times S^{q}\rightarrow S^{q}$ of type $(p_\ast(\zeta), \iota_{q})$, by Whitehead's definition of the $J$-homomorphism we have
\[H\circ J(\zeta)\simeq H\circ F (\tilde{\zeta})\simeq (-1)^{q+1} (p_\ast(\zeta)\ast \iota_{q}).\]
By assumption $J(\zeta)=0$, we then have $\Sigma^{q+1}p_\ast(\zeta)\simeq p_\ast(\zeta)\ast \iota_{q}=0$,
which implies $p_\ast(\zeta)=0$ by Freudenthal suspension theorem if $n\leq 2q-1$. Hence we complete the first step that the sphere bundle (\ref{SVbundle}) admits a section, and then as before there exists a homotopy equivalence $E\simeq (S^q\vee S^n)\cup_h e^{q+n}$.

The next step is to show that the fibre inclusion $i$ of (\ref{SVbundle}) admits a left homotopy inverse, for which we need to show that $\eta: S^{q+n-1}\rightarrow S^q\vee S^n \rightarrow S^q$
is null homotopic. Notice that there is the commutative diagram 
\[
\begin{tikzcd}
\pi_{q+n-1}(S^q) \ar[hook]{r}{i_\ast} \ar[hook]{dr}{j_{1\ast}}  & \pi_{q+n-1}(E) \ar[two heads]{r}{\pi_\ast} & \pi_{q+n-1}(S^n) \\
\pi_{q+n-1}(S^q) \ar[equal]{u} & \pi_{q+n-1}(S^q\vee S^n) \ar{l} \ar{u}{j_{2\ast}},
\end{tikzcd}
\]
where $j_1$ and $j_2$ are inclusions, and the first row is a short exact sequence since the bundle (\ref{SVbundle}) admits a cross section.
Then $i_\ast(\eta)\simeq (j_2\circ j_1)_\ast(\eta)\simeq (j_2)_\ast(h)\simeq 0$ implies that $\eta\simeq 0$. Hence we can form the homotopy commutative diagram 
\[
\begin{tikzcd}
& S^{q+n-1} \ar{d}{h} \ar{dr}{\eta=0} & \\
S^q\ar[hook]{r} \ar{dr}{i} & S^q\vee S^n \ar{d} \ar{r} & S^q\\
& E\ar[dashed]{ru}{r}
\end{tikzcd}
\]
that defines a map $r$ such that $r \circ i\simeq {\rm id}$ since the column is a cofibration sequence. The second step is then completed. 

Now the lemma follows by applying the Whitehead theorem to the map $(r,\pi): E\longrightarrow S^q\times S^n$.
\end{proof}
We can now prove Theorem~\ref{spheregaugeintro}.
\begin{proof}[Proof of Theorem~\ref{spheregaugeintro}]
Since for the homotopy cofibration $S^n\stackrel{s}{\rightarrow}E\stackrel{p}{\rightarrow} X$ appeared in Diagram \ref{defeta} $s$ admits a left inverse $\pi$, there is a short exact sequence of sets $[X, BG]\hookrightarrow [E,BG]\twoheadrightarrow [S^{n}, BG]$.
Hence $[E, BG]\cong[X, BG]$ by the assumption $\pi_{n-1}(G)=0$. The proof of the two decompositions in the theorem is then similar to that of Proposition \ref{Gsplit} based on our discussions on the homotopy type of $\Sigma E$ in Subsection \ref{subsec: spherebundle2} for the first one, and Lemma \ref{spheresplit3} with the fact that $\Sigma (S^q\times S^n)\simeq S^{q+1}\vee S^{n+1}\vee S^{q+n+1}$ for the second one.
\end{proof}

\section{Homotopy decompositions of $E$-gauge groups over $(n-2)$-connected $2n$-manifolds}\label{other2nsec}
In this section, we study the gauge groups over other types of closed oriented $2n$-manifolds in the spirit of Section \ref{complexsec}. We only consider principal bundles of three exceptional Lie groups. The material in this section heavily relies on various computational techniques in unstable homotopy theory.

\subsection{Homotopy of suspended $(n-2)$-connected $2n$-manifolds}
In this subsection, we consider some types of closed oriented $2n$-manifolds ($n\geq4$) which admit a cell decomposition of the form
\begin{equation}\label{otherM}
M\simeq \bigvee_{i=1}^{m} S^{n-1} \cup e^{n+1}_{(1)}\cup e^{n+1}_{(2)}\ldots \cup e^{n+1}_{(m)}\cup e^{2n}.
\end{equation}

By the method in Section \ref{complexsec}, we can determine the homotopy type of the $(n+1)$-skeleton $M_{n+1}$ of $M$ through the homotopy cofibre sequence
$\bigvee_{i=1}^{m} S_j^{n}\stackrel{f}{\rightarrow} \bigvee_{i=1}^{m} S_i^{n-1} \rightarrow M_{n+1}$.
Since $\pi_n(\bigvee_{i=1}^{m} S^{n-1})\cong \bigoplus_{i=1}^{m}\mathbb{Z}/2$, the homotopy type of $M_{n+1}$ is determined by an $(m\times m)$ matrix $C=(c_{i,j})$ over the field $\mathbb{Z}/2$ where $c_{i,j}$ represents the component of the attaching map from $S^{n}_j$ to $S_i^{n-1}$. As in Section \ref{complexsec}, performing row-addition transformations will not change the homotopy type of $M_{n+1}$. Also, the row-multiplying transformations are valid since we work over $\mathbb{Z}/2$. Hence, we can do all row operations. Indeed we can also apply column operations.

\begin{lemma}\label{columnn-1}
Under above assumptions let ${\rm rank} (C)=c$. Then 
\[M_{n+1}\simeq \bigvee_{i=1}^{c}\Sigma^{n-3}\mathbb{C}P^{2}\vee\bigvee_{j=1}^{m-c}(S^{n-1}\vee S^{n+1}).\]
\end{lemma}
\begin{proof}
For any non-singular matrices $P$ and $Q\in {\rm M}_{m\times m}(\mathbb{Z}/2)$, we can construct maps to fill in the following commutative diagram
\[
 \begin{tikzcd}
\bigvee_{j=1}^{m} S^{n} \ar{d}{Q}  \ar{r}{C} & \bigvee_{i=1}^{m} S^{n-1}  \ar[r] \ar{d}{P} & M_{n+1} \ar{d}{\phi}[swap]{\simeq} \\
\bigvee_{j=1}^{m} S^{n} \ar{r}{PCQ^{-1}}  & \bigvee_{i=1}^{m} S^{n-1}  \ar[r]   & M_{n+1}^\prime\simeq M_{n+1},
 \end{tikzcd}
\]
where either row is a homotopy cofibre sequence, and the homotopy equivalence is obtained by applying the Five lemma on the associated long exact sequences of homologies. Hence, we can choose $P$ and $Q$ such that $PCQ^{-1}$ is a diagonal matrix of rank ${\rm rank} (C)$. Since $\pi_n(S^{n-1})=\langle \eta_{n-1}\rangle$ generated by the first Hopf map, the homotopy cofibre of $\eta_{n-1}$ is homotopy equivalent to $\Sigma^{n-3}\mathbb{C}P^{2}$ and then the lemma follows.
\end{proof}

Since we want to study the homotopy of gauge groups through Proposition \ref{Gsplit}, we need to study the homotopy type of $\Sigma M$, which is determined by the top attaching map lying in the image of the composition 
\begin{eqnarray*}
&&\pi_{2n-1}(M_{n+1})\\
&\cong& \pi_{2n-1}\big(\bigvee_{i=1}^{c}\Sigma^{n-3}\mathbb{C}P^{2}\vee\bigvee_{j=1}^{m-c}(S^{n-1}\vee S^{n+1})\big) \\
&\stackrel{}{\cong}& \bigoplus_{i=1}^{c}\pi_{2n-1}(\Sigma^{n-3}\mathbb{C}P^{2}) \oplus\bigoplus_{i=1}^{m-c}\pi_{2n-1}(S^{n-1})
\oplus\bigoplus_{i=1}^{m-c}\pi_{2n-1}(S^{n+1})\oplus W \\
&\stackrel{\Sigma}{\longrightarrow}& \bigoplus_{i=1}^{c}\pi_{2n}(\Sigma^{n-2}\mathbb{C}P^{2}) \oplus\bigoplus_{i=1}^{m-c}\pi_{2n}(S^{n})
\oplus\bigoplus_{i=1}^{m-c}\pi_{2n}(S^{n+2})\\
&\stackrel{}{\hookrightarrow}&  \pi_{2n}(\Sigma M_{n+1}),
\end{eqnarray*}
where the second isomorphism, the factor $W$ and the last inclusion are determined by the Hilton-Milnor theorem (see Section \MyRoman{11}.6 of \cite{Whitehead78} for instance).

\subsection{Homotopy types of suspended $4$-connected $12$-manifolds}
In this subsection, we study the homotopy type of $\Sigma M$ (\ref{otherM}) when $n=6$.
In order to study the suspension of the top attaching map of $M$, we need to compute the homotopy groups of suspended complex projective planes.
\begin{lemma}\label{suspensionsmashCP2split}
There is an isomorphism of homotopy groups
\[
\pi_{i}(S^6\wedge \mathbb{C}P^2 \wedge \mathbb{C}P^2)\cong \pi_{i}(\Sigma^8 \mathbb{C}P^2)
\]
for each $0\leq i\leq 11$.
\end{lemma}
\begin{proof}
First notice that at any odd prime $p$ $S^6\wedge \mathbb{C}P^2 \wedge \mathbb{C}P^2\simeq S^{10}\vee S^{12}\vee S^{12}\vee S^{14}$, $\Sigma^8 \mathbb{C}P^2\simeq S^{10}\vee S^{12}$, and their $i$-th homotopy groups are trivial for each $0\leq i\leq 11$. Hence it remains to consider the $2$-torsion parts of the involved homotopy groups, and they are already in the stable range by Freudenthal suspension theorem.

In the proof of Lemma $3.10$ of \cite{Wu2003} Wu constructed a $2$-local homotopy cofibration
$S^6\stackrel{f}{\rightarrow} \mathbb{C}P^2 \wedge \mathbb{C}P^2 \rightarrow Z$ such that the $7$-skeleton $Z_7$ of $Z$ is homotopy equivalent to $\Sigma^2\mathbb{C}P^2$.
It follows that there is the stable homotopy fibration
$S^{12}\stackrel{}{\rightarrow} S^6\wedge \mathbb{C}P^2 \wedge \mathbb{C}P^2\rightarrow \Sigma^6Z$. By the exact sequence of homotopy groups of fibration, it implies that $\pi_i(S^6\wedge \mathbb{C}P^2\wedge \mathbb{C}P^2)\cong \pi_i(\Sigma^6Z)$ for each $0\leq i\leq 11$. On the other hand, since $Z_7\simeq \Sigma^2\mathbb{C}P^2$, we have $\pi_i(\Sigma^6Z)\cong \pi_{i}(\Sigma^8\mathbb{C}P^2)$. The lemma then follows.
\end{proof}

We now can get the homotopy information of the top attaching map that we need for the homotopy of gauge groups.
\begin{lemma}\label{n=6,n-1}
$\Sigma: \pi_{11}(\Sigma^3\mathbb{C}P^2)\rightarrow \pi_{12}(\Sigma^4\mathbb{C}P^2)\cong \mathbb{Z}/2$ is an epimorphism.
\end{lemma}
\begin{proof}
By Lemma \ref{suspensionsmashCP2split} and Freudenthal suspension theorem, $\pi_{12}(\Sigma^4\mathbb{C}P^2\wedge \Sigma^3\mathbb{C}P^2)\cong\pi_{12}(\Sigma^{9}\mathbb{C}P^2)$. Then by the $EHP$-sequence (Chapter \MyRoman{12} Theorem 2.2 of \cite{Whitehead78}) we have the exact sequence $\pi_{11}(\Sigma^3\mathbb{C}P^2)\stackrel{\Sigma}{\rightarrow}\pi_{12}(\Sigma^4\mathbb{C}P^2)\stackrel{H}{\rightarrow}\pi_{12}(\Sigma^{9}\mathbb{C}P^2)$. In Proposition $8.1$ and Proposition $10.5$ of \cite{Mukai82} Mukai showed that 
$\pi_{12}(\Sigma^{9}\mathbb{C}P^2)=0$ and $\pi_{12}(\Sigma^4\mathbb{C}P^2)=\{i_\ast \nu_6^2\}\cong \mathbb{Z}/2$, where $\pi_{12}(S^6)=\{\nu_6^2\}\cong \mathbb{Z}/2$ and $i$ is the injection of the bottom cell. The lemma now follows.
\end{proof}

\begin{lemma}\label{SigmaM1257}
Let $M$ be a $12$-dimensional manifold with cell decomposition of the form $M\simeq \bigvee_{i=1}^{m} S^{5} \cup e^{7}_{(1)}\cup e^{7}_{(2)}\ldots \cup e^{7}_{(m)}\cup_f e^{12}$.
Then there exists a non-negative number $c$ with $0\leq c\leq m$ such that 
\[\Sigma M\simeq \Sigma Z \vee \bigvee_{i=1}^{c-1} \Sigma^4 \mathbb{C}P^2 \vee\bigvee_{i=1}^{m-c-1} S^6\vee \bigvee_{i=1}^{m-c} S^8,\]
where $Z\simeq \big(\Sigma^3\mathbb{C}P^2 \vee S^5\big)\cup_g e^{12}$ for some $g$.
\end{lemma}
\begin{proof}
This proof is similar to the treatment in Section \ref{complexsec}. With Lemma \ref{columnn-1} ($n=6$), Lemma \ref{n=6,n-1} and the facts that \cite{Toda62} $\pi_{12}(S^8)=0$ and $\pi_{12}(S^6)\cong \mathbb{Z}/2$, we can represent the attaching map $\Sigma f$ of the top cell of $\Sigma M$ by a matrix $B$, and apply allowable matrix transformations to simplify $B$ as in Lemma \ref{trianglelemma}.
The decomposition in the lemma then can be showed similarly as Corollary \ref{Zsplit}.
\end{proof}

\subsection{Homotopy types of suspended $6$-connected $16$-manifolds}
In this subsection, we study the homotopy type of $\Sigma M$ when $n=8$. The argument is similar to that in the last subsection.

\begin{lemma}\label{n=8n-1lemma}
The homomorphism
$\Sigma: \pi_{15}(\Sigma^5\mathbb{C}P^2)\rightarrow \pi_{16}(\Sigma^6\mathbb{C}P^2)\cong \mathbb{Z}/2\oplus\mathbb{Z}/4$ 
is surjective.
\end{lemma}
\begin{proof}
First, by the similar argument at the beginning of the proof of Lemma \ref{n=6,n-1}, we have that $\pi_{16}(\Sigma^6\mathbb{C}P^2\wedge \Sigma^5\mathbb{C}P^2)\cong \pi_{16}(\Sigma^{13}\mathbb{C}P^2)$.
By the $EHP$-sequence (Chapter \MyRoman{12} Theorem 2.2 of \cite{Whitehead78}) and Proposition $8.1$ of \cite{Mukai82}, we have the exact sequence $\pi_{15}(\Sigma^5\mathbb{C}P^2)\stackrel{\Sigma}{\rightarrow}\pi_{16}(\Sigma^6\mathbb{C}P^2)\stackrel{H}{\rightarrow}\pi_{16}(\Sigma^{13}\mathbb{C}P^2)=0$.
We are left to compute $\pi_{16}(\Sigma^6\mathbb{C}P^2)$. First, by Proposition $12.2$ of \cite{Mukai82}, we have
\[
\pi_{17}(\Sigma^7\mathbb{C}P^2)=\{i_\ast \sigma_9 \eta_{16}\} \oplus \{\tilde{\nu}_{11}\nu_{14}\}\cong \mathbb{Z}/2\oplus\mathbb{Z}/4,
\]
where (following Toda \cite{Toda62}) $\pi_{17}(S^9)=\{\sigma_9\eta_{16}\}\oplus \{\bar{\nu}_9\}\oplus \{\epsilon_9\} \cong \mathbb{Z}/2\oplus\mathbb{Z}/2\oplus \mathbb{Z}/2$, and $i$ is the injection of the bottom cell, $\pi_{14}(\Sigma^7\mathbb{C}P^2)=\{\tilde{\nu}_{11}\}\cong \mathbb{Z}/24$ (Proposition $10.2$ of \cite{Mukai82}), and $\pi_{17}(S^{14})=\{\nu_{14}\}\cong \mathbb{Z}/24$. 
Then by the $EHP$-sequence, Lemma \ref{suspensionsmashCP2split} and and Freudenthal suspension theorem, we have the exact sequence
\[\pi_{18}(\Sigma^7\mathbb{C}P^2\wedge\Sigma^6\mathbb{C}P^2)\stackrel{P}{\rightarrow}\pi_{16}(\Sigma^6\mathbb{C}P^2)\stackrel{\Sigma}{\rightarrow}\pi_{17}(\Sigma^7\mathbb{C}P^2)\stackrel{H}{\rightarrow}\pi_{17}(\Sigma^7\mathbb{C}P^2\wedge\Sigma^6\mathbb{C}P^2),\]
and $\pi_{i}(\Sigma^7\mathbb{C}P^2\wedge\Sigma^6\mathbb{C}P^2)\cong \pi_{i}(\Sigma^{15}\mathbb{C}P^2)$ for $i\leq 18$.
Since by Proposition $8.1$ of \cite{Mukai82} $\pi_{18}(\Sigma^{15}\mathbb{C}P^2)=0$, $\pi_{17}(\Sigma^{15}\mathbb{C}P^2)\cong \mathbb{Z}$, and $\pi_{17}(\Sigma^7\mathbb{C}P^2)$ is a torsion group, we have $\pi_{16}(\Sigma^6\mathbb{C}P^2)\cong \pi_{17}(\Sigma^7\mathbb{C}P^2)\cong \mathbb{Z}/2 \oplus \mathbb{Z}/4$, and the lemma follows.
\end{proof}
Using the methods in Section \ref{complexsec}, and combining Lemma \ref{n=8n-1lemma} and the facts \cite{Toda62}
\[\Sigma: \oplus_{i=1}^{3}\mathbb{Z}/2\cong\pi_{15}(S^7)\hookrightarrow \pi_{16}(S^8),  \ \ \  \Sigma:  \pi_{15}(S^9)\stackrel{\cong}{\rightarrow} \pi_{16}(S^{10})\cong \mathbb{Z}/2,\]
we obtain the following lemma analogous to Lemma \ref{SigmaM1257}.
\begin{lemma}\label{SigmaM1679}
Let $M$ be a $16$-dimensional manifold with cell decomposition of the form $M\simeq \bigvee_{i=1}^{m} S^{7} \cup e^{9}_{(1)}\cup e^{9}_{(2)}\ldots \cup e^{9}_{(m)}\cup_f e^{16}$.
Then there exists a non-negative number $c$ with $0\leq c\leq m$ such that 
\[\Sigma M\simeq \Sigma Z \vee \bigvee_{i=1}^{c-4} \Sigma^6 \mathbb{C}P^2 \vee\bigvee_{i=1}^{m-c-3} S^8\vee \bigvee_{i=1}^{m-c-1} S^9,\]
where $Z\simeq \big(\bigvee_{i=1}^{2} \Sigma^5\mathbb{C}P^2\vee \bigvee_{i=1}^{3}S^7\vee S^{9}\big)\cup_g e^{16}$ for some $g$.  ~$\qqed$
\end{lemma}

\subsection{Homotopy decompositions of gauge groups over $(n-2)$-connected $2n$-manifolds}
Let us start with some homotopy information of exceptional Lie groups (Theorem \MyRoman{5}, \cite{BS58}).
\begin{eqnarray*}
&&\pi_i(E_6)=0 \ \ \ \ \ \ {\rm for}~4\leq i \leq 8,\ \ \ \ \ \ \    \pi_9(E_6)=\mathbb{Z},\\
&&\pi_i(E_7)=0 \ \ \ \ \ \ {\rm for}~4\leq i \leq 10,\ \ \ \ \ \ \pi_{11}(E_7)=\mathbb{Z},\\
&&\pi_i(E_8)=0 \ \ \ \ \ \ {\rm for}~4\leq i \leq 14,\ \ \ \ \ \ \pi_{15}(E_8)=\mathbb{Z}.\\
\end{eqnarray*}
With the above facts, the following lemma follows immediately from Lemma \ref{Gindex}.
\begin{lemma}\label{Zindex}
Let $M$ be a $4$-connected closed oriented $m$-manifold. If
\[(m, G)= (10, E_6), (12, E_7) ~{\rm or}~(16, E_8),\]
then $[M, BG]\cong \mathbb{Z}$. ~$\qqed$
\end{lemma}

We can now prove Theorem \ref{lasttheoremintro}. 
\begin{proof}[Proof of Theorem \ref{lasttheoremintro}]
We have already showed the cell structure of $M_{n+1}$ in Lemma \ref{columnn-1}.
For the $E_7$-principal bundles over $M$ when $n=6$, we have $[M, BE_7]\cong \mathbb{Z}$ by Lemma \ref{Zindex}. Then by applying Proposition \ref{Gsplit} for $m=12$ and $k=4$, and Lemma \ref{SigmaM1257} with the fact that $[\Sigma^2 \mathbb{C}P^2, E_7]=0$, 
we can show the homotopy decomposition of the $E_7$-gauge groups. Similarly, for the $E_8$-principal bundles over $M$ when $n=8$, we have $[M, BE_8]\cong \mathbb{Z}$ by Lemma \ref{Zindex}. By applying Proposition \ref{Gsplit} for $m=16$ and $k=6$, and Lemma \ref{SigmaM1679} with the fact that $[\Sigma^4 \mathbb{C}P^2, E_8]=0$, we can show the homotopy decomposition of the $E_8$-gauge groups. The homotopy decomposition of the gauge groups localized away from 2 follows immediately from the integral ones.
\end{proof}

\section{Applications to the homotopy exponent problem}\label{expsec}
In this section, we apply the results in the previous sections to study the homotopy exponent problem. For any pointed space $X$, its $p$-th \textit{homotopy exponent} is the least power of $p$ which annihilates the $p$-torsion in $\pi_\ast(X)$, and denoted by ${\rm exp}_p(X)$ or simply ${\rm exp}(X)$. We are actually interested in the upper bounds of the odd primary homotopy exponents of gauge groups.

\begin{lemma}\label{expformu1}
Suppose the triple ($M$, $G$, $p$) belongs to one of the following groups:
\begin{itemize}
\item Type $A$: $M$ is an $(n-1)$-connected closed oriented $2n$-manifold, and $G$ is a topological group with $\pi_{n-1}(G)\cong \pi_{n}(G)=0$,
\begin{itemize}
\item $n\not\equiv 0~{\rm mod}~4$, and $p$ is odd, or
\item $n\equiv 0~{\rm mod}~4$, $(p, d_{\frac{n}{4}})=1$, where $d_s$ is the denominator of $B_s/4s$ for $B_s$ the $s$-th Bernoulli number;
\end{itemize}
\item Type $B$: $M=E$ is the total space of the sphere bundle $S^q\stackrel{i}{\longrightarrow} E\stackrel{\pi}{\longrightarrow} S^n$ of an oriented real vector bundle with a cross section such that $J(\zeta)=0$, $G$ is topological group such that $\pi_{n-1}(G)=\pi_{q-1}(G)=0$ and $q=n+2k$ with $k\geq 0$;
\item Type $C$: $M$ is an $(n-2)$-connected closed oriented $2n$-manifold, and $p$ is odd. $(G, n)=(E_7, 6)$, or $(G, n)=(E_8, 8)$. 
\end{itemize}
Then in the Type $A$ and Type $C$ cases, we have for any $\alpha\in [M, BG]=\pi_{2n-1}(G)$ 
\[
{\rm exp}(\mathcal{G}_\alpha(M))\leq {\rm max}\{{\rm exp}(\mathcal{G}_\alpha(S^{2n})), {\rm exp}(G)\};
\]
and in the Type $B$ case, we have for any $\alpha\in [M, BG]=\pi_{2(n+k)-1}(G)$ 
\[
{\rm exp}(\mathcal{G}_\alpha(M))\leq {\rm max}\{{\rm exp}(\mathcal{G}_\alpha(S^{2(n+k)})), {\rm exp}(G)\}.
\]
\end{lemma}
\begin{proof}
It suffices to show that in the Type $A$ and Type $C$ cases, there exist $p$-local homotopy decompositions of the form $\mathcal{G}_\alpha(M)\simeq_{(p)} \mathcal{G}_\alpha(S^{2n})  \times \prod_{n_j}\Omega^{n_j}G$, and in the Type $B$ case $\mathcal{G}_\alpha(M)\simeq_{(p)} \mathcal{G}_\alpha(S^{2(n+k)}) \times \prod_{n_j}\Omega^{n_j}G$,
where the numbers $n_j$ depend on the case. However, these decompositions follow immediately from Theorems \ref{generalMn-12nintro}, \ref{spheregaugeintro} and \ref{lasttheoremintro}. 
\end{proof}

From Lemma \ref{expformu1}, we are leaded to consider the homotopy exponents of the gauge groups over even dimensional spheres. For that we only consider $\mathcal{G}_\alpha(S^{2n})$ for simplicity.
There is a homotopy fibre sequence
\begin{equation}\label{gaugefibforexp}
\Omega(\Omega^{2n-1}_0G)\rightarrow \mathcal{G}_\alpha(S^{2n}) \rightarrow G\stackrel{\partial_\alpha}{\rightarrow}\Omega^{2n-1}_0G\rightarrow B\mathcal{G}_\alpha(S^{2n}) \stackrel{{\rm ev}}{\rightarrow} BG,
\end{equation}
where the connecting map $\partial_{\alpha}$, by a lemma of Lang \cite{Lang73}, is identified with the Samelson product $\langle \alpha, {\rm id}_{G}\rangle$. 
This suggests that the exponents of $\mathcal{G}_\alpha(S^{2n})$ are controlled by that of $G$.
Thus we need to consider the homotopy exponents of the structure groups, and Lie groups are our main concern here. 

To set notations, let $X$ be a finite $p$-local $H$-space. Then by a classical result of Hopf, there is a rational homotopy equivalence 
\[
X\simeq_{(0)} S^{2n_1+1}\times S^{2n_2+1}\times \cdots \times S^{2n_l+1},
\]
where the index set $\mathfrak{t}(X)=\{n_1, n_2,\ldots, n_l\}$ ($n_1\leq n_2\leq \ldots\leq n_l$) is called the \textit{type} of $X$. Denote also $l(X)= n_l$, $\mathfrak{t}_i(X)=\{k\in \mathfrak{t}(X)~|~k\equiv i~{\rm mod}~(p-1)\}$.
The types of compact simply connected simple Lie groups are well known and summarised in Table \ref{tabletypelie}.
\begin{table}[!htbp]
\caption{Types of simple Lie groups}
\begin{tabular}{lp{3.7cm}|lp{3.7cm}lp{3.7cm}}
\hline
$G$      & Type  &   $G_2$        &   $1, 5$      \\ \hline      
$SU(n)$            & $1, 2, \ldots, n-1$      &   $F_4$        &   $1, 5, 7, 11$       \\ \hline
$Sp(n)$            & $1, 3, \ldots, 2n-1$          &   $E_6$        &   $1, 4, 5, 7, 8, 11$         \\ \hline
$Spin(2n)$            & $1, 3, \ldots, 2n-3, n-1$           &   $E_7$        &   $1, 5, 7, 9, 11, 13, 17$      \\ \hline
$Spin(2n+1)$            & $1, 3, \ldots, 2n-1$   &    $E_8$        &   $1, 7, 11, 13, 17, 19, 23, 29$             \\ \hline
\end{tabular}
\label{tabletypelie}
\end{table}

Due to a classical result of Serre \cite{Serre53}, it is well known that these Lie groups can be decomposed into a product of spheres at large primes, which was generalized to finite $p$-local $H$-spaces by Kumpel \cite{Kumpel72}.

\begin{theorem}\label{regulardecomH}
Let $X$ be a finite ${\rm mod}~p$ $H$-space of type $\{n_1=1, n_2,\ldots, n_l\}$. If $H^\ast(X;\mathbb{Z})$ is $p$-torsion free and $p\geq l(X)+1$, then there exists a map 
\[
f: S^{2n_1+1}\times S^{2n_2+1}\times \cdots \times S^{2n_l+1}\rightarrow X,
\]
which is a ${\rm mod}~p$ homotopy equivalence.~$\qqed$
\end{theorem}

The primes $p$ for which there exists a map $f$ as in Theorem \ref{regulardecomH} are called \textit{the regular primes} of $X$. With this terminology, Theorem \ref{regulardecomH} just claims that $X$ is $p$-regular if $p\geq l(X)+1$.

On the other hand, thanks to a construction of Gray \cite{Gray69} on a family of infinitely many elements of order $p^n$ in $\pi_\ast(S^{2n+1})$, Cohen, Neisendorfer and Moore in their famous work \cite{Cohen79, Neisen3} showed that 
\begin{equation}\label{sphereexp}
{\rm exp}(\Omega^i_0S^{2n+1})=p^n,
\end{equation}
for any $i\geq0$ and odd prime $p$. The following lemma then follows immediately.
\begin{lemma}\label{regularexplie}
Let $G$ be a Lie group in Table \ref{tabletypelie} with $p$-torsion free cohomology. Then if $p\geq l(G)+1$, we have ${\rm exp}(\Omega^i_0G)=p^{l(G)}$ for any $i\geq0$. ~$\qqed$
\end{lemma}

Combining Lemma \ref{expformu1} and Lemma \ref{regularexplie}, we can get the following proposition.
\begin{proposition}\label{expformu2}
Let $(G, M, p)$ as in Lemma \ref{expformu1}, and $G$ be a Lie group in Table \ref{tabletypelie} with $p$-torsion free cohomology. If $p\geq l(G)+1$, we have 
\[
\hspace{2.3cm}
{\rm exp}(\mathcal{G}_\alpha(M))\leq {\rm max}\{{\rm exp}(\mathcal{G}_\alpha(S^{2n})), p^{l(G)}\}\leq p^{2l(G)}. \hspace{2.3cm}\Box
\] 
\end{proposition}

In order to obtain a more refined estimate, we consider some intermediate primary decomposition of Lie groups. 
As a generalization of classical work of Mimura, Nishida and Toda \cite{MNT77}, Theriault \cite{Theriault07} showed a deep result about homotopy decompositions of the low rank Lie groups in Table \ref{tablelie2} with further applications of the homotopy exponents. 
\begin{table}[!htbp]
\caption{Some low rank Lie groups}
\begin{tabular}{l|p{3.7cm}lp{3.7cm}}
\hline 
$SU(n)$            & $n-1\leq (p-1)(p-2)$        \\ \hline
$Sp(n)$            & $2n\leq (p-1)(p-2)$          \\ \hline
$Spin(2n+1)$            & $2n\leq (p-1)(p-2)$         \\ \hline
$G_2, F_4, E_6$            & $p\geq5$       \\ \hline
$E_7, E_8$            & $p\geq7$       \\ \hline
\end{tabular}
\label{tablelie2}
\end{table}

\begin{theorem}[Theriault \cite{Theriault07}]\label{Theriaultlieexp}
Let $G$ be a Lie group in Table \ref{tablelie2}. Then there is a homotopy decomposition 
\[
G\simeq_{(p)} B_1\times B_2\times \cdots \times B_{p-1},
\]
where $B_i$ is a homotopy associative, homotopy commutative, spherically resolved $H$-space for each $1\leq i\leq p-1$. Furthermore, there exists a composition 
\[S^{2l(B_i)+1}\stackrel{c}{\rightarrow} B_i\stackrel{q}{\rightarrow} S^{2l(B_i)+1}\]
of degree $p^{r_i(G)}$ for some $r_i(G)\in \mathbb{N}$, where $c$ is a representative of a generator of $\pi_{2l(B_i)+1}(B_i)$, and $q$ is the map to the highest dimensional resolving sphere of $B_i$, and a homotopy commutative diagram of $H$-maps
\[
\xymatrix{
B_i \ar[r]^{p^{r_i(G)}}  \ar[d]^{}   & B_i\ar@{=}[d]\\
\prod\limits_{k\in \mathfrak{t}_i(G)} S^{2k+1} \ar[r]  & B_i,
}
\]
for each $1\leq i\leq p-1$. In particular, with $r(G)={\rm max}\{r_1(G), r_2(G), \ldots, r_{p-1}(G)\}$
\[\hspace{2.3cm}{\rm exp}(B_i)\leq p^{r_i(G)+l(B_i)}, \ \ \  \ \ \ \  {\rm exp}(G) \leq p^{r(G)+l(G)}.\hspace{2.3cm}\Box\]
\end{theorem}
For a given collection of sets $\mathcal{S}=\{S_1, S_2,\ldots, S_{p-1}\}$ we define $S_{pn+i}=S_i$ for any $1\leq i\leq p-1$ and $n\in \mathbb{Z}$. Equivalently, $\mathcal{S}$ is regarded as indexed by $\mathbb{Z}/(p-1)$ as in \cite{KKT13}.
\begin{proposition}\label{expformu3}
Let $(G, M, p)$ be as in Lemma \ref{expformu1}, and $G$ be a Lie group in Table \ref{tablelie2} with $p$-torsion free cohomology. Denote $l^{(n)}_i(G)=l(B_i)+l(B_{i-n})$, $r^{(n)}_i(G)=r_i(G)+r_{n-i}(G)$, 
and $s^{(n)}(G)={\rm max}\{r^{(n)}_i(G)+l^{(n)}_i(G)~|~1\leq i\leq p-1\}$.
Then if $n-1\in \mathfrak{t}(G)$, we have 
\[
{\rm exp}(\mathcal{G}_\alpha(M))\leq p^{s^{(n)}(G)}.
\]
\end{proposition}
\begin{proof}
It was shown in \cite{KKT13} that there exists a homotopy equivalence 
\[
\mathcal{G}_\alpha(S^{2n})\simeq_{(p)} \mathfrak{B}_1^\alpha\times  \mathfrak{B}_2^\alpha\times\cdots\times \mathfrak{B}_{p-1}^\alpha,
\]
where $\mathfrak{B}_{i}^\alpha$ for each $i\in\mathbb{Z}/p-1$ is the total space of a homotopy fibre sequence 
\begin{equation}\label{KKTfibre}
\Omega(\Omega^{2n-1}_0B_i)\rightarrow\mathfrak{B}_{i}^\alpha\rightarrow B_{i-n}.
\end{equation}
The proposition then follows from Theorem \ref{Theriaultlieexp} and Lemma \ref{expformu1}.
\end{proof}

There are two special cases of Proposition \ref{expformu3} for which we can obtain better upper bounds of the homotopy exponents. Denote ${\rm ord}(\alpha)$ to be the least number of powers of $p$ such that $p^{{\rm ord}(\alpha)}\alpha=0$.
\begin{proposition}\label{expformu4}
Let $(G, M, p)$ as in Lemma \ref{expformu1}, and $G$ be a Lie group in Table \ref{tablelie2} with $p$-torsion free cohomology. Suppose $n-1\in \mathfrak{t}(G)$, then
\begin{itemize}
\item[(1).] ${\rm exp}(\mathcal{G}_\alpha(M))\leq p^{l(G)+{\rm ord}(\alpha)}$ if $p\geq l(G)+1$;
\item[(2).] ${\rm exp}(\mathcal{G}_\alpha(M))= p^{l(G)}$ if $p\geq l(G)+n+1$.
\end{itemize}
\end{proposition}
\begin{proof}
For Case (1), by (\ref{gaugefibforexp}) we have the exact sequence 
\[
\pi_\ast(\Omega(\Omega_0^{2n-1}G))\rightarrow \pi_\ast(\mathcal{G}_\alpha(S^{2n}))\rightarrow {\rm Ker}(\partial_{\alpha\ast})\rightarrow 0,
\]
and it is easy to show that $p^{{\rm ord}(\alpha)} \pi_\ast(G)\subseteq  {\rm Ker}(\partial_{\alpha\ast})$.
On the other hand, we notice that under the condition $p\geq l(G)+1$, the fibration (\ref{KKTfibre}) is now a homotopy fibration over a sphere with fiber an iterated loop space of a sphere. Moreover, by Theorem $3.12$ of \cite{KKT13} we have the homotopy commutative diagram of fibre sequences for each $i\in \mathbb{Z}/p-1$
\[
\xymatrix{
\Omega(\Omega^{2n-1}_0B_i) \ar[r] \ar[d]  &\mathfrak{B}_{i}^\alpha   \ar[r] \ar[d] & B_{i-n} \ar[d]\ar[r]^{\partial_\alpha^{(i)}\ \ \ }  &\Omega^{2n-1}_0B_i \ar[d]\\
\Omega(\Omega^{2n-1}_0G) \ar[r]  &\mathcal{G}_\alpha(S^{2n})  \ar[r]  & G  \ar[r]^{\partial_\alpha \ \ \ } & \Omega^{2n-1}_0G,
}
\]
where $\partial_\alpha^{(i)}$ is the $i$-th component of $\partial_\alpha$. 
It follows that $p^{{\rm ord}(\alpha)} \pi_\ast(B_{i-n})\subseteq  {\rm Ker}(\partial_{\alpha^{(i)}\ast})$.
Then the order $|{\rm coker}(\pi_\ast(\mathfrak{B}_{i}^\alpha)\rightarrow \pi_\ast(B_{i-n}))|\leq p^{{\rm ord}(\alpha)}$. By Lemmas $2.2$ and $2.3$ of \cite{Theriault04}, we see that ${\rm exp}(\mathfrak{B}_{i}^\alpha)\leq p^{{\rm ord}(\alpha)}\cdot {\rm max}({\rm exp}(B_i),{\rm exp}(B_{n-i}))$. The inequality (1) then follows from the $p$-regularity of $G$ and Lemma \ref{expformu1}.

For Case (2), as shown in Proposition $3.14$ of \cite{KKT13}, the gauge group of a sphere is decomposed completely as $\mathcal{G}_\alpha(S^{2n})\simeq_{(p)}\prod\limits_{k\in \mathfrak{t}(G)}(S^{2k+1}\times \Omega_0^{2n}S^{2k+1})$.
The equality (2) then follows immediately from Lemma \ref{expformu1} and (\ref{sphereexp}).
\end{proof}

\begin{remark}
We close the section with several remarks.
\begin{itemize}
\item As shown in \cite{Theriault07}, $r(SU(n))=\nu_p((n-1)!)\leq  \lfloor \frac{n-2}{p-1} \rfloor$, which is much smaller than $l(SU(n))=n-1$. Moreover, since $SU(2n)\simeq_{(p)} Sp(n)\times SU(2n)/Sp(n)$, and $Spin(2n+1)\simeq_{(p)} Sp(n)$ (Harris, \cite{Harris61}), we can obtain upper bounds of the homotopy exponents of $Sp(n)$ and $Spin(2n+1)$ from that of $SU(n)$. The homotopy exponents of exceptional Lie groups are known, as summarised in Theorem $1.10$ of \cite{DT08}. 
\item Proposition \ref{expformu4} also holds for the exceptional Lie groups with special $(n ,p)$ as listed in Proposition $3.14$ of \cite{KKT13}.
\item We can work out concrete examples based on the above results. For instance, for the two cases of Type $C$, we have 
for the $E_7$-bundle with $p\geq 29$, ${\rm exp}(\mathcal{G}_k(M))= p^{17}$, and for the $E_8$-bundle with $p\geq 41$, ${\rm exp}(\mathcal{G}_k(M))= p^{29}$.
\end{itemize}
\end{remark}

\noindent\textbf{Acknowledgements.} 
The author is indebted to Haibao Duan for introducing several points of Wall's work on $(n-1)$-connected $2n$-manifolds. He would like to thank Juno Mukai for the computations of homotopy groups in Lemma \ref{n=8n-1lemma}.
He benefitted much from the class of Thomas Farrell on geometric topology at the Yau Mathematical Sciences Center of Tsinghua University.
He is also grateful to Daisuke Kishimoto and Stephen Theriault for helpful discussions.

The author sincerely thanks Stephen Theriault for proofreading this paper. 
He also wants to thank the anonymous referee most warmly for careful reading of the manuscript and numerous suggestions that have improved the exposition of this paper.

\end{document}